\documentclass[11pt]{amsart}
\usepackage{latexsym,amsmath,amssymb,amsthm,mathrsfs,amscd,enumerate,latexsym,array}
\usepackage[all]{xy}

\theoremstyle{definition}
\newtheorem{Unity}{Unity}[section]
\newtheorem{dfn}[Unity]{Definition}

\newtheorem{exam}[Unity]{Example}
\newtheorem{case}{Case}
\newtheorem*{ack}{Acknowledgements}

\theoremstyle{plain}
\newtheorem{thm}[Unity]{Theorem}
\newtheorem{prop}[Unity]{Proposition}
\newtheorem{conj}[Unity]{Conjecture}
\newtheorem{lem}[Unity]{Lemma}

\let\oldenumerate\enumerate
\renewcommand{\enumerate}{
\oldenumerate
\setlength{\itemsep}{3pt}
\setlength{\leftskip}{-13pt}
}

\let\olditemize\itemize
\renewcommand{\itemize}{
\olditemize
\setlength{\itemsep}{2pt}
\setlength{\leftskip}{-15pt}
}

\begin{document}

\title{An invariant for embedded Fano manifolds covered by linear spaces}
\author{Taku Suzuki}
\address{Department of Mathematics, School of Fundamental Science and Engineering, Waseda University, 3-4-1 Ohkubo, Shinjuku, Tokyo 169-8555, Japan}
\email{s.taku.1231@moegi.waseda.jp}

\maketitle

\begin{abstract}
For an embedded Fano manifold $X$, we introduce a new invariant $S_X$ related to the dimension of covering linear spaces. 
The aim of this paper is to classify Fano manifolds $X$ which have large $S_X$. 
\end{abstract}

\section{Introduction}

We consider an embedded Fano manifold $X$, i.e., a smooth complex projective variety $X \subset \mathbb{P}^N$ with ample anti-canonical line bundle $-K_X$. 
If $X$ is covered by lines, then we consider a family $H_1$ of lines on $X$ through a fixed general point $x$. 
We can regard $H_1$ as a subvariety of $\mathbb{P}(T_{x} X^{\vee})$. 
If $H_1$ is also a Fano manifold covered by lines, then we can consider a family $H_2$ of lines on $H_1$ through a fixed general point. 
In this way, we inductively define an \textit{$m$-th order family $H_m$ of lines}. 
In addition, we define an invariant $S_X$ as the greatest number $m$ such that there is an $m$-th order family of lines. 
As shown later, any Fano manifold $X$ satisfies $S_X \le {\rm dim}\,X$, and $X$ is covered by linear spaces of dimension at least $S_X$. 
In particular, $S_X = {\rm dim}\,X$ holds if and only if $X$ is a linear space. 

The aim of this paper is to classify Fano manifolds $X$ of Picard number $1$ which have large $S_X$. 
We will prove the following theorem: 

\begin{thm}\label{MainThm}
Let $X \subset \mathbb{P}^N$ be a Fano manifold of Picard number $1$ and dimension $n \ge 2$. 
\begin{enumerate}
\item If $S_X > \frac{n}{2}$, then $X$ is a linear space $\mathbb{P}^n$. 
\item If $S_X = \frac{n}{2}$, then $X$ is isomorphic to either a quadric hypersurface $Q^{2m}$ or the Grassmannian $G(2,\mathbb{C}^{m+2})$ $(n=2m \ge 4)$. 
\item Assume that Conjecture \ref{Conj} is true.  If $S_X = \frac{n-1}{2}$, then $X$ is isomorphic to one of the following: 
\begin{itemize}
\item[(a)] a quadric hypersurface $Q^{2m+1}$ $(n=2m+1)$;
\item[(b)] the symplectic Grassmannian $SG(2,\mathbb{C}^{m+3})$ $(n=2m+1)$;
\item[(c)] a cubic hypersurface in $\mathbb{P}^4$ $(n=3)$; 
\item[(d)] a complete intersection of two quadric hypersurfaces in $\mathbb{P}^5$ $(n=3)$; 
\item[(e)] a $3$-dimensional linear section of the Grassmannian $G(2,\mathbb{C}^{5})$ in $\mathbb{P}^9$ $(n=3)$. 
\end{itemize}
\end{enumerate}
\end{thm}

\begin{conj}\label{Conj}
Let $X \subset \mathbb{P}^N$ be a Fano manifold of Picard number $1$ which is covered by lines. 
Assume that a family of lines on $X$ through $x$ embedded in $\mathbb{P}(T_{x} X^{\vee})$ is projectively equivalent to $\mathbb{P}(\mathscr{O}_{\mathbb{P}^1}(2) \oplus \mathscr{O}_{\mathbb{P}^1}(1)^{m-1})$ embedded in $\mathbb{P}^{2m}$ by the tautological line bundle for a general point $x \in X$. 
Then $X$ is isomorphic to the symplectic Grassmannian $SG(2,\mathbb{C}^{m+3})$. 
\end{conj}

Projective $n$-folds covered by linear spaces of large dimension $m$ have been investigated by several authors. 
E.\ Sato studied the case $m \ge \frac{n}{2}$. 
We immediately obtain the first two arguments (1) and (2) of Theorem \ref{MainThm} from the following theorem: 

\begin{thm}[{\cite{Sa}, see also \cite[Corollary 5.3]{NO}}]\label{Thm1}
Let $X$ be a smooth projective $n$-fold $X \subset \mathbb{P}^N$. 
Assume that $X$ is covered by linear spaces of dimension $m \ge \frac{n}{2}$. 
Then $X$ is one of the following: 
\begin{itemize}
\item[(i)] a linear $\mathbb{P}^k$-bundle $(k \ge m)$;
\item[(ii)] a quadric hypersurface $Q^{2m}$ $(m=\frac{n}{2})$;
\item[(iii)] the Grassmannian $G(2,\mathbb{C}^{m+2})$ $(m=\frac{n}{2})$.
\end{itemize}
\end{thm}

C.\ Novelli and G.\ Occhetta classified the case $m=\frac{n-1}{2}$ under the assumption that the normal bundles of the covering linear spaces are numerically effective \cite[Theorem 1.1]{NO}. 
However, the complete classification of the case $m=\frac{n-1}{2}$ is not known. 
So, the argument (3) of Theorem \ref{MainThm} is the main part of this paper. 

In order to classify the case $S_X = \frac{n-1}{2}$, we need Conjecture \ref{Conj}, i.e., a characterization of sympletic Grassmannians in terms of families of lines. 
G.\ Occhetta, L.\ E. Sol\'a Conde, and K.\ Watanabe's recent result \cite{OSW} says that Conjecture \ref{Conj} is ture if the assumption holds for every point $x \in X$. 

This paper is organized as follows. 
In section 2, we mention several facts concerning families of lines. 
In section 3, we introduce higher order families of lines and the invariant $S_X$. 
In section 4, we prove (3) of Theorem \ref{MainThm}. 

\begin{ack}
The author would like to express his gratitude to his supervisor Professor Hajime Kaji for valuable discussions. 
The author would also like to thank Professor Kiwamu Watanabe for beneficial suggestions and comments. 
The author is also grateful to Professor Yasunari Nagai and Professor Daizo Ishikawa for useful comments. 
\end{ack}

\section{Preliminaries}

We denote by $\rho_X$ the Picard number of a manifold $X$. 
We refer to \cite[I and II]{Ko} for basic theory of families of rational curves. 
Throughout this paper, we consider a smooth projective Fano $n$-fold $X$ embedded in a projective space $\mathbb{P}^N$. 
Let $x$ be a fixed general point of $X$. 
We denote by $F_1(X,x)$ the Hilbert scheme of lines on $X$ passing through $x$.
From now on, let $H$ be an irreducible component of $F_1(X,x)$, which is called a \textit{family of lines through $x$}. 

\begin{lem}[{\cite[Theorem 1.3 and Proposition 1.5]{Hw1}}]\label{Lem1}
We can regard $H$ as a smooth projective subvariety of the projectivised tangent space $\mathbb{P}(T_{x} X^{\vee})$ which has dimension $(-K_X \cdot H)-2$, where $\mathbb{P}(T_{x} X^{\vee}):=(T_{x} X \setminus \{0\}) \slash \mathbb{C}^{\times} \cong \mathbb{P}^{n-1}$, and $(-K_X \cdot H)$ means the anti-canonical degree of a line parametrized by $H$. 
\end{lem}

\begin{exam}[{see \cite[1.4]{Hw1}}]\label{Exam1}
\ 
\begin{enumerate}
\item If $X$ is a linear space $\mathbb{P}^n$, then $H$ is equal to $\mathbb{P}(T_{x} X^{\vee})$. 
\item If $X$ is a quadric hypersurface $Q^n$ in $\mathbb{P}^{n+1}$ ($n \ge 3$), then $H$ is a quadric hypersurface $Q^{n-2}$ in $\mathbb{P}(T_{x} X^{\vee})$. 
\item If $X$ is $\mathbb{P}^k\times \mathbb{P}^m$ (Segre embedding), then $H$ is a linear space $\mathbb{P}^{k-1}$ or $\mathbb{P}^{m-1} $ in $\mathbb{P}(T_{x} X^{\vee})$. 
\item If $X$ is the Grassmannian $G(k,\mathbb{C}^m)$ (Pl{\"u}cker embedding), then $H$ is $\mathbb{P}^{k-1}\times \mathbb{P}^{m-k-1}$ embedded in $\mathbb{P}(T_{x} X^{\vee})$ by the Segre embedding. 
\item If $X$ is the symplectic Grassmannian $SG(k,\mathbb{C}^m)$, then $H$ is $\mathbb{P}(\mathscr{O}_{\mathbb{P}^{k-1}}(2)\oplus \mathscr{O}_{\mathbb{P}^{k-1}}(1)^{m-2k})$ embedded in $\mathbb{P}(T_{x} X^{\vee})$ by the tautological line bundle. 
\item If $X$ is a complete intersection of degrees $(d_1,\ldots ,d_m)$ in $\mathbb{P}^N$ ($N \ge \sum d_i+2$), then $H$ is a complete intersection of degrees $(2,3,\ldots, d_1,\ldots , 2,3,\ldots, d_m)$ in $\mathbb{P}(T_{x} X^{\vee})$. 
\end{enumerate}
\end{exam}

\begin{lem}\label{Lem3}
If $H$ contains a linear space $L$ of dimension $m$, then $X$ contains a linear space of dimension $m+1$ through $x$. 
\end{lem}

\begin{proof}
Let $U$ be the universal family of $H$, and let $\pi : U \rightarrow H$ and $e: U \rightarrow X$ be the associated morphisms. 
By the proof of \cite[Lemma 2.3]{AC}, we have that the restriction of $e$ to $\pi^{-1}(L)$ is the blow-up of $\mathbb{P}^{m+1}$ and $e(\pi^{-1}(L))$ is a linear space of dimension $m+1$ through $x$. 
\end{proof}

\begin{lem}\label{Lem2}
We have the following: 
\begin{enumerate}
\item If ${\rm dim}\,H=n-1$ $($i.e.,\ $H=\mathbb{P}(T_{x} X^{\vee}))$, then $X$ is a linear space $\mathbb{P}^n$. 
\item If $\rho_X=1$ and ${\rm dim}\,H=n-2$, then $X$ is a quadric hypersurface ${Q}^n$ $(n \ge 3)$. 
\item If $\rho_X=1$ and ${\rm dim}\,H=n-3$, then $X$ is one of the following: 
\begin{itemize}
\item[(i)] a cubic hypersurface;
\item[(ii)] a complete intersection of two quadric hypersurfaces; 
\item[(iii)] the Grassmannian $G(2,\mathbb{C}^5)$ or a linear section of it. 
\end{itemize}
\end{enumerate}
\end{lem}

\begin{proof}
(1) follows from Lemma \ref{Lem3}. 
If $\rho_X=1$, then $-K_X$ is linearly equivalent to $\mathscr{O}_{\mathbb{P}^N}({\rm dim}\,H+2)|_X$ by Lemma \ref{Lem1}. 
So, (2) follows from \cite{KO}, and (3) follows from \cite{Fu1} and \cite{Fu2}. 
\end{proof}

\begin{lem}\label{Lem4}
If $\rho_X=1$ and ${\rm dim}\,H \ge \frac{n-1}{2}$, then $H \subset \mathbb{P}(T_{x} X^{\vee})$ is non-degenerate. 
\end{lem}

\begin{proof}
By \cite[Theorem 1.3 and Proposition 1.5]{Hw1}, $H$ is contained in a smooth projective subvariety $V$ of $\mathbb{P}(T_{x} X^{\vee})$ whose every irreducible component has same dimension as $H$, and \cite[Theorem 2.5]{Hw1} implies that $V \subset \mathbb{P}(T_{x} X^{\vee})$ is non-degenerate. 
Since ${\rm dim}\,H \ge \frac{1}{2}{\rm dim}\,\mathbb{P}(T_{x} X^{\vee})$, we have $H=V$, thus we obtain the conclusion. 
\end{proof}

\begin{lem}\label{Lem5}
If $\rho_X=1$ and $H \subsetneq \mathbb{P}(T_{x} X^{\vee})$ is a linear space of positive dimension, then ${\rm dim}\,H \le \frac{n-4}{2}$. 
\end{lem}

\begin{proof}
Set $d:={\rm dim}\,H (\ge 1)$. 
Lemma \ref{Lem3} implies that $X$ is covered by linear spaces of dimension $d+1(\ge 2)$. 
By the proof of \cite[Proposition 2.1]{Hw2}, we see that the normal bundles of the linear spaces are trivial. 
On the other hand, if $d \ge \frac{n-3}{2}$ (i.e., $d+1 \ge \frac{n-1}{2}$), then the normal bundles cannot be trivial by \cite[Proposition 6.4]{NO}. 
Thus we obtain $d \le \frac{n-4}{2}$. 
\end{proof}

\section{Higher order families of lines}

In \cite{Su}, we introduced \textit{higher order minimal families of rational curves} and invariants $\overline{N}_X$ and $\underline{N}_X$ for Fano manifolds $X$. 
We introduce line versions of them as follows: 

\begin{dfn}
We write $X \models H$ if $X$ is an embedded Fano manifold covered by lines and $H$ is a family of lines through a fixed general point. 
If $X \models H_1$ and $H_1$ is also a Fano manifold covered by lines in the projectivised tangent space, then we consider  $H_1 \models H_2$ a family of lines on $H_1$ through a fixed general point. 
In the same way, if there is a chain 
$$X \models H_1 \models H_2 \models \cdots \models H_m$$
of length $m$, we call $H_m$ an \textit{$m$-th order family of lines}. 
In addition, we denote by $S_X$ the greatest number $m$ such that there is an $m$-th order family of lines. 
\end{dfn}

\begin{exam}\label{Exam2}
By Example \ref{Exam1}, we have the following: 
\begin{enumerate}
\item If $X$ is a linear space $\mathbb{P}^n$, then 
$$X \models \mathbb{P}^{n-1} \models \mathbb{P}^{n-2} \models \cdots \models \mathbb{P}^1 \models \text{pt},$$
so $S_X=n$. 
\item If $X$ is a quadric hypersurface $Q^n$, then 
$$X \models Q^{n-2} \models Q^{n-4} \models \cdots \models \begin{cases} Q^2 \models \text{pt}& \text{ if }n\text{ is even}\\Q^3 \models Q^1& \text{ if }n\text{ is odd}\end{cases},$$
so $S_X=\begin{cases} \frac{n}{2}& \text{ if }n\text{ is even}\\\frac{n-1}{2}& \text{ if }n\text{ is odd}\end{cases}$. 
\item If $X$ is the Grassmannian $G(2,\mathbb{C}^{m+2})$ (${\rm dim}\,X=2m$), then 
$$X \models \mathbb{P}^1\times \mathbb{P}^{m-1} \models \begin{cases}\text{pt}\\ \mathbb{P}^{m-2} \models \mathbb{P}^{m-3} \models \cdots \models \text{pt}\end{cases},$$
so $S_X=m$. 
\item If $X$ is the symplectic Grassmannian $SG(2,\mathbb{C}^{m+3})$ (${\rm dim}\,X=2m+1$), then 
$$X \models \mathbb{P}(\mathscr{O}_{\mathbb{P}^1}(2)\oplus \mathscr{O}_{\mathbb{P}^1}(1)^{m-1}) (\cong \text{Bl}_{\mathbb{P}^{m-2}}\mathbb{P}^m) \models \mathbb{P}^{m-2} \models \mathbb{P}^{m-3} \models \cdots \models \text{pt},$$
so $S_X=m$. 
\end{enumerate}
\end{exam}

\begin{prop}\label{Prop1}
Any Fano manifold $X\subset \mathbb{P}^N$ is covered by linear spaces of dimension at least $S_X$. 
\end{prop}

\begin{proof}
Set $m:=S_X$. 
By definition of $S_X$, there is a chain 
$$X(=:H_0) \models H_1 \models \cdots \models H_m$$
of length $m$. 
By applying Lemma \ref{Lem3}, we inductively see that $H_i$ is covered by linear spaces of dimension $m-i$. 
\end{proof}

\begin{exam}\label{Exam3}
There is a Fano manifold $X \subset \mathbb{P}^N$ which is covered by linear spaces of dimension greater than $S_X$. 
For example, let $X$ be a complete intersection of two quadric hypersurfaces in $\mathbb{P}^7$. 
Then $X \models H$ is a complete intersection of two quadric hypersurfaces in $\mathbb{P}^4$. 
Since $H$ contains a line, Lemma \ref{Lem3} yields that $X$ is covered by linear spaces of dimension $2$. 
However, $H$ is not covered by lines, so $S_X=1$. 
\end{exam}

\begin{prop}\label{Prop2}
Let $X \subset \mathbb{P}^N$ be a Fano $n$-fold. 
\begin{enumerate}
\item $0 \le S_X \le n$ holds. 
\item If $S_X=n$, then $X$ is a linear space $\mathbb{P}^n$. 
\item If $n \ge 2$ and $S_X=n-1$, then $(X,\mathscr{O}_{\mathbb{P}^N}(1)|_X)$ is isomorphic to either of the following: 
\begin{itemize}
\item[(i)] $(\mathbb{P}^1\times \mathbb{P}^{n-1}, \mathscr{O}(d,1))$ for some $d\ge 1$;
\item[(ii)] $(\mathbb{P}(\mathscr{O}_{\mathbb{P}^1}(d+1)\oplus {\mathscr{O}_{\mathbb{P}^1}(d)}^{n-1}), \mathscr{O}(1))$ for some $d\ge 1$, where $\mathscr{O}(1)$ is the tautological line bundle.
\end{itemize}
\end{enumerate}
\end{prop}

\begin{proof}
(1) and (2) follow from Proposition \ref{Prop1}. 
We assume $S_X=n-1 \ge 1$. 
Then, Proposition \ref{Prop1} and Theorem \ref{Thm1} imply that $X$ is a linear $\mathbb{P}^{n-1}$-bundle, so we get a vector bundle $\mathscr{E}$ on $\mathbb{P}^1$ such that $(X,\mathscr{O}_{\mathbb{P}^N}(1)|_X)$ is isomorphic to $(\mathbb{P}(\mathscr{E}), \mathscr{O}_{\mathbb{P}(\mathscr{E})}(1))$. 
We can write $\mathscr{E} \cong \bigoplus _{i=1}^n \mathscr{O}_{\mathbb{P}^1}(a_i)$ for some positive integers $a_1 \le \cdots \le a_n$. 
Since $X\cong \mathbb{P}(\mathscr{E})$ is a Fano manifold, we see that $\sum_{i=1}^n a_i \le n a_1+1$. 
Thus we obtain either (i) or (ii). 
\end{proof}

\begin{prop}[(1) and (2) of Theorem \ref{MainThm}]
Let $X \subset \mathbb{P}^N$ be a Fano $n$-fold of Picard number $1$. 
\begin{enumerate}
\item If $S_X > \frac{n}{2}$, then $X$ is a linear space $\mathbb{P}^n$. 
\item If $S_X = \frac{n}{2}$, then $X$ is either a quadric hypersurface $Q^{2m}$ or the Grassmannian $G(2,\mathbb{C}^{m+2})$ $(n=2m \ge 4)$. 
\end{enumerate}
\end{prop}

\begin{proof}
The arguments immediately follow from Proposition \ref{Prop1} and Theorem \ref{Thm1}. 
\end{proof}

\section{Main result}

In this section, we prove (3) of Theorem \ref{MainThm}: 

\begin{thm}
Let $X \subset \mathbb{P}^N$ be a Fano manifold of Picard number $1$ and dimension $n=2m+1\ge 3$. 
Assume that Conjecture \ref{Conj} is true and $S_X = m$. 
Then $X$ is isomorphic to one of the following: 
\begin{itemize}
\item[(a)] a quadric hypersurface $Q^{2m+1}$;
\item[(b)] the symplectic Grassmannian $SG(2,\mathbb{C}^{m+3})$;
\item[(c)] a cubic hypersurface in $\mathbb{P}^4$ $(m=1)$; 
\item[(d)] a complete intersection of two quadric hypersurfaces in $\mathbb{P}^5$ $(m=1)$; 
\item[(e)] a $3$-dimensional linear section of the Grassmannian $G(2,\mathbb{C}^{5})$ in $\mathbb{P}^9$ $(m=1)$.
\end{itemize}
\end{thm}

\begin{proof}
Since $S_X=m$, there is a chain 
$$X(=:H_0) \models H_1 \models \cdots \models H_m$$
of length $m$. 
Set $n_i:={\rm dim}\,H_i$. 
Since $X$ is not a linear space, Lemma \ref{Lem2}(1) yields $n_0-n_1 \ge 2$. 
Let $P_i$ be the projectivised tangent space of $H_{i-1}$ containing $H_i$. 
Note that $P_i$ is isomorphic to a projective space of dimension $n_{i-1}-1$. 

\begin{case}
$H_i$ is not a linear space in $P_i$ for any $1 \le i \le m-1$. 
\end{case}

In this case, $n_i-n_{i+1} \ge 2$ for any $1 \le i \le m-1$, so we have: 
$$n_0-n_1 \le (2m+1)-2(m-1)=3.$$
Thus, Lemma \ref{Lem2} implies that $X$ is one of the following: 
\begin{itemize}
\item[(i)] a quadric hypersurface $Q^{2m+1}$;
\item[(ii)] a cubic hypersurface in $\mathbb{P}^{2m+2}$; 
\item[(iii)] a complete intersection of two quadric hypersurfaces in $\mathbb{P}^{2m+3}$; 
\item[(iv)] a $3$-dimensional linear section of the Grassmannian $G(2,\mathbb{C}^{5})$ in $\mathbb{P}^9$ ($m=1$);
\item[(v)] a hyperplane section of the Grassmannian $G(2,\mathbb{C}^{5})$ in $\mathbb{P}^9$ ($m=2$).
\end{itemize}
According to the above inequality, if $n_0-n_1 = 3$ and $m \ge 2$, then $n_1-n_2=2$ must holds. 
However, in case (ii) (resp.\ (iii)), if $m \ge 2$, then $H_1$ is a complete intersection of a cubic hypersuface and a quadric hypersurface (resp.\ two quadric hypersufaces) in $\mathbb{P}^{2m}$, so we cannot take $H_2$ when $m=2$, and $n_1-n_2 = 4$ (resp.\ $3$) when $m \ge 3$. Thus, in case (ii) (resp.\ (iii)), we obtain $m=1$, so $X$ is (c) (resp.\ (d)). 
Notice that (v) is isomorphic to (b) $SG(2,\mathbb{C}^5)$. 

\begin{case}
$H_i$ is a linear space in $P_i$ for some $1 \le i \le m-1$. 
\end{case}

In this case, $H_{i+j}(=P_{i+j})$ has dimension $n_i-j$ for any $j \ge 1$, and $H_m$ is a point, so $n_i=m-i$. 
By taking the minimum $i$, we may assume that $H_{i-1}$ is not a linear space in $P_{i-1}$. 

We show $\rho_{H_{i-1}} \ge 2$. 
We assume by contradiction that $\rho_{H_{i-1}} =1$.  
Then Lemma \ref{Lem5} implies $\frac{n_{i-1}-4}{2} \ge n_i (=m-i)$, so
$$2m+1 \ge 2(i-1)+n_{i-1} \ge 2(i-1)+2(m-i)+4=2m+2,$$
a contradiction. 
Thus we obtain $\rho_{H_{i-1}} \ge 2$. 
In particular, $i \ge 2$ and $m \ge 3$. 
Moreover, since $H_{i-1}$ is embedded in a projective space $P_{i-1}$ of dimension  $n_{i-2}-1$, we have $\frac{n_{i-2}}{2} \ge n_{i-1}(\ge 2+m-i)$ by Barth and Larsen's theorem \cite{BL}. 

We know that none of the manifolds in Lemma \ref{Lem2} satisfies the condition of this case, so $n_0-n_1 \ge 4$. 
Hence, if $i \ge 3$, then
$$2m+1 \ge 4+2(i-3)+n_{i-2} \ge 4+2(i-3)+2(2+m-i) = 2m+2,$$
a contradiction. 
Thus we get $i=2$, and the inequalities $\frac{n_0}{2} \ge n_1\ge m$ yield $n_1=m$. 
By applying Lemma \ref{Lem4}, we have that $H_1 \subset P_1$ is non-degenerate. 
Furthermore, since $S_{H_1}=m-1=n_1-1$, Proposition \ref{Prop2} implies that $(H_1,\mathscr{O}_{P_1}(1)|_{H_1})$ is isomorphic to either of the following: 
\begin{itemize}
\item[(i)] $(\mathbb{P}^1\times \mathbb{P}^{m-1}, \mathscr{O}(d,1))$ for some $d\ge 1$;
\item[(ii)] $(\mathbb{P}(\mathscr{O}_{\mathbb{P}^1}(d+1)\oplus {\mathscr{O}_{\mathbb{P}^1}(d)}^{m-1}), \mathscr{O}(1))$ for some $d\ge 1$, where $\mathscr{O}(1)$ is the tautological line bundle.
\end{itemize}
Since $H_1$ is embedded in a projective space $P_1$ of dimension $2m$, in both cases, we have $d=1$ by the following Lemma \ref{Lem6}. 
In case (i), the linear span of $H_1$ in $P_1$ has dimension $2m-1$, so $H_1 \subset P_1$ is degenerate, that is a contradiction. 
Therefore, the case (i) does not occur. 
On the other hand, in case (ii), we obtain that $X$ is isomorphic to (b) $SG(2,\mathbb{C}^{m+3})$ by applying Conjecture \ref{Conj}. 
\end{proof}

\begin{lem}[{see \cite[Lemma 3.1]{Ca}}]\label{Lem6}
\ 
\begin{enumerate}
\item Let $S$ be $\mathbb{P}^1\times \mathbb{P}^{m-1}$ embedded in $\mathbb{P}^{dm+m-1}$ by $\mathscr{O}(d,1)$. 
If $d\ge 2$, then the secant variety of $S$ has dimension $2m+1$. 
\item Let $S$ be $\mathbb{P}(\mathscr{O}_{\mathbb{P}^1}(d+1)\oplus {\mathscr{O}_{\mathbb{P}^1}}(d)^{m-1})$ embedded in $\mathbb{P}^{dm+m}$ by the tautological line bundle. 
If $d\ge 2$, then the secant variety of $S$ has dimension $2m+1$. 
\end{enumerate} 
\end{lem}


\begin{thebibliography}{1}
\bibitem{AC}
C.\ Araujo and A-M.\ Castravet, {Polarized minimal families of rational curves and higher Fano manifolds}, \textit{American J.\ Math.}, \textbf{134}(1) (2012), 87-107. 

\bibitem{BL}
W.\ Barth and M.\ E.\ Larsen, {On the homotopy groups of complex projective algebraic manifolds}, \textit{Math. Scand.}, \textbf{30} (1972), 88-94. 

\bibitem{Ca}
M.\ L.\ Catalano-Johnson, {The possible dimension of the higher secant varieties}, \textit{Amer.\ J.\ Math.}, \textbf{118} (1996), 355-361. 

\bibitem{Fu1}
T.\ Fujita, On the structure of polarized manifolds with total deficiency one, I, \textit{J.\ Math.\ Soc.\ Japan.} \textbf{32} (1980), 709-725. 

\bibitem{Fu2}
T.\ Fujita, On the structure of polarized manifolds with total deficiency one, II, \textit{J.\ Math.\ Soc.\ Japan.} \textbf{33} (1981), 415-434. 

\bibitem{Hw1}
J.-M.\ Hwang, {Geometry of minimal rational curves on Fano manifolds}, in \textit{School on Vanishing Theorems and Effective Results in Algebraic Geometry} (Trieste,\ 2000), ICTP Lect.\ Notes, vol.\ 6, 335-393. Abdus Salam Int. Cent. Theoret. Phys. (2001).

\bibitem{Hw2}
J.-M.\ Hwang, {Deformation of holomorphic maps onto Fano manifolds of second and fourth Betti numbers 1}, \textit{Ann.\ Inst.\ Fourier}, \textbf{57}(3) (2007), 815-823. 

\bibitem{Ke}
S.\ Kebekus, {Characterizing the projective space after Cho, Miyaoka and Shepherd-Barron}, \textit{Complex Geometry (G{\"o}ttingen, 2000)}, Springer-Verlag, Berlin (2002), 147-155.

\bibitem{KO}
S.\ Kobayashi and T.\ Ochiai, {Characterizations of complex projective spaces and hyperquadrics},
\textit{J.\ Math.\ Kyoto Univ.}, \textbf{13} (1973) 31-47.

\bibitem{Ko}
J.\ Koll\'ar, \textit{Rational Curves on Algebraic Varieties}, Ergeb. Math. Grenzgeb., vol. 32. Springer, Berlin (1996).

\bibitem{NO}
C.\ Novelli and G.\ Occhetta, {Projective manifolds containing a large linear subspace with nef normal bundle}, \textit{Michigan Math.\ J.}, \textbf{60} (2011), 441-462. 

\bibitem{OSW}
G.\ Occhetta, L.\ E.\ Sol\'a Conde, and K.\ Watanabe, {A characterization of symplectic Grassmannians}, \textit{Math.\ Z.}, published online (2016). 

\bibitem{Sa}
E.\ Sato, {Projective manifolds swept out by large dimensional linear spaces}, \textit{Tohoku Math.\ J.}, \textbf{49} (1997), 299-321.  

\bibitem{Su}
T.\ Suzuki, {Higher order minimal families of rational curves and Fano manifolds with nef Chern characters}, arXiv:1606.09350. 

\end{thebibliography}
\end{document}